\documentclass{article}

\usepackage[utf8]{inputenc}
\usepackage{graphicx}
\usepackage[all]{xy}
\usepackage{amssymb}
\usepackage{amsmath}
\usepackage{amsthm}
\usepackage{verbatim}
\usepackage{latexsym}
\usepackage{mathrsfs}

\usepackage{mathrsfs}

\newenvironment{myItemize}{
  \begin{list}{\raisebox{2.2pt}{$\centerdot$}}{%
      \setlength\leftmargin{18pt}
      \setlength\labelwidth{20pt}
    }
  }{
  \end{list}
}

\makeindex

\newcommand{\sval}[1]{\left[ #1 \right]} 

\renewcommand{\Cup}{\bigcup}
\renewcommand{\Cap}{\bigcap}

\newcommand{\g}{\gamma}
\newcommand{\e}{\varepsilon}

\renewcommand{\o}{\omega}

\newcommand{\es}{\varnothing}

\renewcommand{\H}{\mathcal{H}}

\newcommand{\N}{\mathbb{N}} 
\newcommand{\Q}{\mathbb{Q}} 
\newcommand{\R}{\mathbb{R}} 

\newcommand{\Tr}{\operatorname{Tr}}

\newcommand{\rest}{\!\restriction\!}
\newcommand{\restl}{\restriction}  
\newcommand{\dom}{\operatorname{dom}}
\newcommand{\ran}{\operatorname{ran}}

\renewcommand{\le}{\leqslant}

\newcommand{\Sii}{{\Sigma_1^1}}
\newcommand{\Pii}{{\Pi_1^1}}

\swapnumbers
\newtheorem*{Thm*}{Theorem}
\newtheorem{Thm}{Theorem}

\newtheorem{Prop}[Thm]{Proposition}
\newtheorem{Fact}[Thm]{Fact}

\theoremstyle{definition}
\newtheorem{Def}[Thm]{Definition}

\newtheorem{claim}{Claim}[Thm]
\newtheorem{subclaim}{Subclaim}[claim]

\theoremstyle{remark}

\newcommand{\proofvpara}{\text{}}

\newenvironment{proofVOf}[1] {\noindent \textit{Proof of #1.}\ignorespaces\renewcommand{\proofvpara}{\text{#1}}}
{\nopagebreak\hspace*{\fill}\mbox{$\square_{\,\proofvpara}$}\\}

\author{Vadim Kulikov \\ Kurt Gödel Research Center, Vienna, Austria}
\title{The Class of Purely Unrectifiable Sets in $\ell_2$ is $\Pii$-complete}

\date{\today}

\allowdisplaybreaks

\begin{document}

\maketitle

\begin{abstract}
  The space $F(\ell_2)$ of all closed subsets of $\ell_2$ is a Polish space.
  We show that the subset $P\subset F(\ell_2)$ consisting of the purely
  $1$-unrectifiable sets is $\Pii$-complete.
\end{abstract}

\section{Introduction}

The concepts of unrectifiable and purely unrectifiable sets are central in contemporary
geometric measure theory, see e.g. \cite{mattila1999geometry}. '
In some sense these are sets which are not capturable by smooth approximations:
a set is \emph{unrectifiable}, if it cannot be covered (upto a negligible set) by countably many
$C^1$-curves and \emph{$1$-purely unrectifiable}, if its $1$-dimensional Hausdorff measure 
restricted to any $C^1$-curve is zero. 
We only consider $1$-purely unrectifiable 
sets in this article (as opposed to $m$-purely unrectifible for $m>1$),
so we skip the ``$1$'' from the notation. There are several open question concerning 
(partial) characterisations of purely unrectifiable sets such as for example whether
or not the two-dimensional Brownian motion is purely unrectifiably with 
probability~$1$~\cite{PreissVerbal}.

Here we show that the notion of pure unrectifiability is subtle to the extend that
any decision procedure for deciding whether a given closed subset of $\ell_2$ 
is purely unrectifiable or not requires an exhaustive search through continuum many 
cases, that is to say, in the language of
descriptive set theory, the set of all closed purely unrectifiable subsets of $\ell_2$ 
is $\Pii$-hard. On the other hand there \emph{is} a decision procedure of this sort,
so the set is $\Pii$-complete (or \emph{coanalytic complete}).

\paragraph*{Acknowledgment}
I would like to thank David Preiss for introducing me to the concept of unrectifiable sets
and pointing to this research direction.

\section{Basic Definitions}

In order to define purely unrectifiable sets in $\ell_2$, let us review the
definition of $C^1$-curve in~$\ell_2$:

\begin{Def}
  A \emph{Fr\'echet derivative} of a function $f\colon\sval{0,1}\to\ell_2$ 
  at point $x\in\sval{0,1}$ is a linear operator $A_x\colon \R\to\ell_2$
  such that
  $$\lim_{h\to 0}\frac{\|f(x+h)-f(x)-A_xh\|_2}{|h|}=0.$$
  The function belongs to $C^1$, if the Fr\'echet derivative exists at every point
  and the map $x\mapsto A_x$ is continuous
  in the operator norm.   

  The linear operator $A_x$ is uniquely determined by the vector
  $A_x(1)$, so denote $f'(x)=A_x(1)$. Also denote the space of all
  $C^1$-curves by $C^1(\sval{0,1},\ell_2)$.
\end{Def}

\begin{Def}
  A subset $N$ of $\ell_2$ is \emph{purely unrectifiable}, if it is null on every $C^1$-curve. 
  That is, given a $C^1$-map $f\colon \sval{0,1}\to\ell_2$,
  the one-dimensional Hausdorff measure of $N\cap\ran(f)$, denoted 
  $\H^1(N\cap\ran(f))$, equals~$0$.
  Denote the set of purely unrectifiable curves in $\ell_2$ by $P$.
\end{Def}

\section{Preliminaries in Descriptive Set Theory}\label{sec:PrelDST}

We follow the notation and presentation of
the book ``Classical Descriptive Set Theory'' by A. Kechris~\cite{Kech} and 
refer frequently to it below when addressing well-known facts.

A \emph{Polish space} is a separable topological space which is homeomorphic to a complete
metric space. The Hilbert space $\ell_2$ is an example of a Polish space.
A \emph{standard Borel space} is a set $X$ endowed with a $\sigma$-algebra
$S$ such that there exists a Polish topology on $X$ in which the Borel sets are precisely
the sets in~$S$. 

Let $F(\ell_2)$ denote the set of all closed subsets of $\ell_2$.
This is a standard Borel space where the
$\sigma$-algebra is generated by the sets of the form
$$\{A\in F(\ell_2)\mid A\cap U\ne\es\}\eqno(B)$$
where $U$ ranges over the basic open sets of $\ell_2$~\cite[Thm 12.6]{Kech}.
We need the following fact.
Let $H$ be the Hilbert cube $H=\sval{0,1}^{\N}$. By \cite[Thm 4.14]{Kech},
$\ell_2$ can be embedded into $H$ so that the image is a $G_\delta$ subset. Let $e$
be that embedding. Let $K(H)$ be the set of all
compact non-empty subsets of $H$ equipped with the Hausdorff metric; $K(H)$ is a compact Polish space.

\begin{Fact}\label{fact:HilbertCube}
  The embedding $e\colon\ell_2\to H$ induces an embedding of $F(\ell_2)$ into $K(H)$
  such that the image of $F(\ell_2)$ is $G_\delta$ in $K(H)$ thus inducing a Polish topology
  on $F(\ell_2)$~\cite[Thm 3.17]{Kech}. This topology gives rise to the same Borel sets as $(B)$ above. \qed
\end{Fact}

By $\o$ and by $\N$ we denote the set of natural numbers, by $\N_+$ the set of positive natural numbers. 
For $n\in \N$, $\o^n$ is the set of
all functions from $\{0,\dots,n-1\}$ to $\o$, $\o^{<\o}=\Cup_{n\in\N}\o^n$ and $\o^\o$ denotes
the set of all functions from $\o$ to $\o$. Similarly $2^\o$ denotes the set of all functions
from $\o$ to $\{0,1\}$ and $2^{<\o}$ the set of functions from $\{0,\dots,n-1\}$ to $\{0,1\}$
for all $n$. The spaces $\o^\o$ and $2^\o$ are Polish spaces in the product topology.

The set $\o^{<\o}$ can be ordered in a natural way: $p<q$ if $q\rest\dom p=p$. This is 
an example of a \emph{tree}. The set of all trees, $\Tr$ is the set of all downward closed
suborderings of $\o^{<\o}$. The space $\Tr$ can be endowed naturally with a Polish topology as
a closed subset of $2^{\o^{<\o}}$ which is in turn homeomorphic to $2^\o$ via a bijection
$\o\to\o^{<\o}$. A \emph{branch} of a tree $T\in\Tr$
is a sequence $(p_n)_{n<\o}$ such that $p_n\in \o^n$, $p_n<p_{n+1}$ and $p_n\in T$ 
for all $n$. 

A subset of a Polish space $A\subset X$ is $\Sii$, if there is a Polish space $Y$
and a Borel subset $B\subset X\times Y$ such that $A$ is the projection of $B$
to~$X$. A set is $\Pii$ if it is the complement of a $\Sii$ set.

\begin{Def}
  A set $A\subset X$ is \emph{Borel Wadge-reducible} to another $B\subset Y$
  ($X$~and $Y$ are Polish), if there exists a Borel function $f\colon X\to Y$
  such that for all $x\in X$, $x\in A\iff f(x)\in B$. We denote this by $A\le_W B$.

  A set $A\subset X$ is \emph{$\Pii$-hard}, if every $\Pii$ set $B$ is Wadge-reducible to 
  it, $B\le_W A$. Similarly $\Sii$-hard. A set is $\Pii$-\emph{complete} ($\Sii$-complete),
  if it is $\Pii$ and $\Pii$-hard ($\Sii$ and $\Sii$-hard).
\end{Def}

Since the classes $\Sii$ and $\Pii$ are closed under preimages in Borel maps \cite[Thm 14.4]{Kech},
it is clear that if $A$ is $\Sii$ and $B\le_W A$, then $B$ is also $\Sii$. On the other hand
a simple diagonalisation argument together with the Souslin's Theorem~\cite[Thm 14.11]{Kech} shows 
that there are $\Pii$ sets that are not $\Sii$. Therefore a $\Pii$-hard set
cannot be $\Sii$, because it Wadge reduces some $\Pii$ set that is not $\Sii$. In particular
it cannot be Borel.

An example of a $\Pii$-complete set
is the set of those trees in $\Tr$
which do not have a branch~\cite[27.1]{Kech}.
To sum up, the main conclusions in this paper are based on the following two facts:

\begin{Fact}\label{fact:Basics}
  \begin{enumerate}
  \item   If $A$ is $\Pii$-hard and $A\le_W B$, then $B$ is $\Pii$-hard.
  \item   The set $\{T\in \Tr\mid T\text{ has no branches}\}$ is $\Pii$-hard.
    \cite[p. 209]{Kech}\qed
  \end{enumerate}
\end{Fact}

\section{Main Theorem}

\begin{Prop}\label{prop:isPii}
  The set $P=\{A\in F(\ell_2)\mid A\text{ is purely unrectifiable}\}$ is $\Pii$.
\end{Prop}
\begin{proof}
  The space $C^1(\ell_2)$ is Polish in the topology given by the $\sup$-norm.
  Let $A\subset F(\ell_2)\times C^1(\ell_2)$ be the set of those pairs
  $(C,\g)$ such that $\H^1(C\cap\ran\g)>0$. Then the projection of $A$ to the
  first coordinate is precisely the complement of~$P$, so it remains to show
  that $A$ is Borel.
  
  Fix a dense countable subset $D$ of $\ell_2$ and
  define a basic open set of $\ell_2$ to be an open ball $B(x,r)$ where $r\in\Q$
  and $x\in D$. Clearly this is a countable basis.
  
  Since $C\cap\ran\g$ is compact, the inequality $H^1(C\cap\ran\g)>0$ is equivalent
  to the statement that there exists $n\in\N$ such that for all 
  finite sequences $(B(x_1,r_1),\dots,B(x_k,r_k))$ of basic open sets of $\ell_2$,
  if $\sum_{i=1}^k r_i<1/n$, then $C\cap\ran\g\not\subset\overline{\Cup_{i=1}^k B(x_i,r_i)}$.
  Denoting 
  $$A^*(x_1,\dots,x_k,r_1,\dots,r_k)=\{(C,\g)\in F(\ell_2)\times C^1(\ell_2)\mid C\cap\ran\g\not\subset\overline{\Cup_{i=1}^k B(x_i,r_i)}\},$$
  we get
  $$A=\Cup_{n\in\N}\Cap_{k\in\N}\Cap_{{\bar x\in D^k, \bar r\in\Q^k}\atop{r_1+\dots+r_k<1/n}}A^*(x_1,\dots,x_k,r_1,\dots,r_k)$$
  Being a subset of a closed set is Borel, so $A^*(x_1,\dots,x_k,r_1,\dots,r_k)$ is Borel.
  Hence $A$ is Borel.
\end{proof}

\begin{Thm}[Main Theorem]\label{thm:Pii}
  The set $P=\{A\in F(\ell_2)\mid A\text{ is purely unrectifiable}\}$ is $\Pii$-complete.
\end{Thm}
\begin{proofVOf}{Theorem~\ref{thm:Pii}}
  We already showed (Theorem~\ref{prop:isPii}) that $P$ is $\Pii$, so we want to show
  that it is $\Pii$-hard.
  The proof is reminiscent of the proof of \cite[Thm 27.6, pp. 210--211]{Kech}.

  We will show that the set $NB$ of those trees $T\in \Tr$ which do not have a branch is
  Wadge-reducible to $P$. 
  That is, we will find a 
  Borel function $H\colon \Tr\to F(\ell_2)$ such that $H(T)$ is not purely 
  unrectifiable if and only if $T$ has a branch. The result follows then
  from Fact~\ref{fact:Basics}.

  A Cantor set $C\subset\R$ with a positive Lebesgue measure can
  be constructed by removing an open interval of length $1/4$ 
  from the middle of the closed unit interval $\sval{0,1}$, then removing open intervals of length $1/16$ from the 
  middle of each of the remaining
  intervals and so on. At the $n$:th step we have a disjoint union of $2^n$ closed intervals. From left to
  right, label these intervals by $C^1_n,\dots,C^{2^n}_n$ and set $C=\Cap_{n=0}^{\infty}\Cup_{k=1}^{2^n}C^k_n$.

  Let $\{e_{n,m}\mid n,m\in\N\}$ be a basis for $\ell_2$. 
  For each $s\in \o^{<\o}$ let us define 
  a finite subset $v_s$ of $\ell_2$ as follows:
  $$v_s=\Big\{\sum_{n=0}^{\dom(s)-1}\frac{1+p(n)}{\sqrt{2^n}} e_{n,s(n)}\mid p\in 2^{\dom(s)}\Big\}.$$
  Then for every tree $T\in \Tr$, let
  $$H(T)=\overline{\Cup_{s\in T}v_s}.$$

  \begin{claim}\label{cl1}
    If $T\in \Tr$ has a branch, then there is a $C^1$-function $f\colon \sval{0,1}\to\ell_2$ such that
    the one-dimensional Hausdorff measure of $H(T)\cap \ran f$ is positive.
  \end{claim}
  \begin{proofVOf}{Claim~\ref{cl1}}
    Suppose that $T$ has a branch and that $b\in \o^\o$ is such that $b\rest n\in T$ for all $n$.
    Let us construct a $C^1$-function $f\colon \sval{0,1}\to\ell_2$ as follows. 
    For $n\in\N$ define $f_n\colon\sval{0,1}\to\R$ to be a smooth function such that
    \begin{myItemize}
    \item $f_n(x)=\frac{1}{\sqrt{2^{n}}}$ for $x\in C^k_n$ when $k$ is odd, 
      and $f_n(x)=\frac{2}{\sqrt{2^{n}}}$ for $x\in C^k_n$ when $k$ is even. 
    \item range of $f_n$ is $\sval{\frac{1}{\sqrt{2^{n}}},\frac{2}{\sqrt{2^{n}}}}$
    \item if $I$ is an open interval which is removed at the $k$:th stage in 
      the construction of $C$, and $x\in I$,
      then
      $$0<f'_n(x)\le \frac{4^{k+1}}{\sqrt{2^{n}}}.$$
    \end{myItemize}
    The derivative can be bounded in this way because
    if $I$ is an open interval that is removed at the $k$:th stage, 
    then $|I|=4^{-k}$ and in this interval, the function is only required to either raise from 
    $1/\sqrt{2^n}$ to $2/\sqrt{2^n}$ or decrease the same amount in the opposite direction.
    On the other hand, if $x\in C$, then the derivative of $f_k$ is $0$ for all $k$. 

    Now let $f(x)=\sum_{n=0}^\infty f_n(x)e_{n,b(n)}$. Clearly $f(x)\in\ell_2$ for all $x$:
    \begin{eqnarray*}
      \|f(x)\|_2^2&=&\sum_{n=0}^\infty |f_n(x)|^2 \\
      &\le&\sum_{n=0}^\infty |\frac{2}{\sqrt{2^{n}}}|^2 \\    
      &=&\sum_{n=0}^\infty \frac{2}{2^{n}} \\    
      &=&4.
    \end{eqnarray*}

    \begin{subclaim}\label{sbcl1}
      The function $f$ has a Fr\'echet derivative at each $x\in\sval{0,1}$.
    \end{subclaim}
    \begin{proofVOf}{Subclaim~\ref{sbcl1}} 
      The vector $A_x=\sum_{n=0}^\infty f'_n(x)e_{n,b(n)}$ is in $\ell_2$, because the absolute value
      of $f_n'(x)$ is bounded by $\frac{4^{k+1}}{\sqrt{2^{n}}}$ where $k$ is 
      a constant natural number which depends on $x$.
      Thus, $A_x$ defines a bounded linear operator $h\mapsto A_xh$. We claim that $A_x$ is the Fr\'echet 
      derivative of $f$ at $x$. For that we need to show that 
      $$\lim_{h\to 0}\frac{\|f(x+h)-f(x)-A_xh\|_2}{|h|}=0.$$
      So assume that $\e>0$. The numerator can be rewritten as
      $$\sqrt{\sum_{n=0}^\infty |f_n(x+h)-f_n(x)-f_n'(x)h|^2}.$$
      Let us show first that there exists $k\in\N$ such that for all $h$
      $$\sum_{n=k}^\infty |f_n(x+h)-f_n(x)-f_n'(x)h|^2\le \e^2h^2:$$
      \begin{eqnarray*}
        |f_n(x+h)-f_n(x)-f_n'(x)h|^2&\le&(|f_n(x+h)-f_n(x)|+|f_n'(x)h|)^2\\
        \text{(mean value theorem) }      &=&(|f_n'(\xi)||h|+|f_n'(x)||h|)^2\\
                                      &=&(|f_n'(\xi)|+|f_n'(x)|)^2h^2\\
    (\text{for some constant } K)        &\le&\left(\frac{K}{2^n}\right)^2h^2.
      \end{eqnarray*}
      The last inequality follows from the definition of $f$.
      Therefore for each $i\in\N$ we have
      $$\sum_{n=i}^\infty|f_n(x+h)-f_n(x)-f_n'(x)h|^2\le \sum_{n=i}^\infty\left(\frac{K}{2^n}\right)^2h^2.$$
      Now, by choosing $k$ big enough we can make sure that
      $\sum_{n=k}^{\infty}\left(\frac{K}{2^n}\right)^2<\e^2$,
      so pick this $k$.      
      Then, for each $n<k$, let $h_n>0$ be small enough real number
      such that $|f_n(x+h_n)-f_n(x)-f_n'(x)h_n|\le \displaystyle\frac{\e}{2^n} h_n$
      and let $h=h_\e=\min_{n<k}h_n$. Then we have:
      \begin{eqnarray*}
        \frac{\|f(x+h)-f(x)-A_xh\|_2}{|h|}&=&\frac{\sqrt{\sum_{n=0}^\infty |f_n(x+h)-f_n(x)-f_n'(x)h|^2}}{|h|}\\
                                        &\le&\frac{\sqrt{\Big(\sum_{n=0}^{k-1} |f_n(x+h)-f_n(x)-f_n'(x)h|^2\Big)+\e^2 h^2}}{|h|}\\
                                        &\le&\frac{\sqrt{\Big(\sum_{n=0}^{k-1} (\frac{\e}{2^n} h)^2\Big)+\e^2 h^2}}{|h|}\\
                                          &<&\frac{\sqrt{4\e^2 h^2+\e^2 h^2}}{|h|}\\
                                          &=&\sqrt{5}\e.
      \end{eqnarray*}
    \end{proofVOf}

    \begin{subclaim}\label{sbclFrDerC}
      The Fr\'echet derivative of $f$ is continuous. Thus $f\in C^1(\sval{0,1},\ell_2)$.
    \end{subclaim}
    \begin{proofVOf}{Subclaim \ref{sbclFrDerC}}
      Let $x\in \sval{0,1}$ and $\e>0$. Denote by $A_x$ the Fr\'echet derivative of $f$ at $x$,
      which has the following form by the previous proof:
      $$A_x=\sum_{n=0}^\infty f'_n(x)e_{n,b(n)}.$$
      The norm of a linear operator from $\R$ to $\ell_2$ (such as $A_x$)
      is determined by the norm of the value at $1$, thus for example 
      $$\|A_x\|=\|A_x(1)\|_2=\sum_{n=0}^\infty |f'_n(x)|^2.$$
      So for every $y\in\sval{0,1}$ we have:
      \begin{eqnarray*}
        \|A_x-A_y\|&=&\Big\|\sum_{n=0}^\infty (f'_n(x)-f'_n(y))e_{n,b(n)}\Big\|_2\\
                   &=&\sqrt{\sum_{n=0}^\infty |f'_n(x)-f'_n(y)|^2}.
      \end{eqnarray*}
      Now similarly as in the previous proof, let us find $k\in\N$ such that
      $$\sum_{n=k}^\infty |f'_n(x)-f'_n(y)|^2<\e^2.$$
      But 
      $$|f'_n(x)-f'_n(y)|^2\le (|f'_n(x)|+|f'_n(y)|)^2\le \left(\frac{K}{2^n}\right)^2,$$
      where $K$ is some constant (this follows again from the definition of $f$). So 
      we can find a big enough $k$ as required. Now, for every $i<k$ pick
      $\delta_i$ such that for every $y$ in the $\delta_i$-neighbourhood of $x$ we have
      $|f'_n(x)-f'_n(y)|<\e/2^n$. This is possible since $f_n$ are smooth by definition.
      Then let $\delta=\min_{i<k}\delta_i$. Now, if $y$ is the $\delta$-neighbourhood of $x$, then 
      by applying the above, we have
      \begin{eqnarray*}
        \|A_x-A_y\|&=&\sqrt{\sum_{n=0}^\infty |f'_n(x)-f'_n(y)|^2}\\
                   &=&\sqrt{\Big(\sum_{n=0}^{k-1} |f'_n(x)-f'_n(y)|^2\Big)+\sum_{n=k}^\infty |f'_n(x)-f'_n(y)|^2}\\
                 &\le&\sqrt{\Big(\sum_{n=0}^{k-1} |f'_n(x)-f'_n(y)|^2\Big)+\e^2}\\
                 &\le&\sqrt{\Big(\sum_{n=0}^{k-1} (\e/2^n)^2\Big)+\e^2}\\
                 &<&\sqrt{2\e^2+\e^2}\\
                 &=&\sqrt{3}\e.           
      \end{eqnarray*}
    \end{proofVOf}
    
    \begin{subclaim}\label{sbclhomeo}
      $f$ is a homeomorphism onto its image.
    \end{subclaim}
    \begin{proofVOf}{Subclaim \ref{sbclhomeo}}
      Since $\dom f$ is compact, it is sufficient to show that it is injective.
      Let $x,y\in \sval{0,1}$. If there is an interval $I$ which
      is removed at some stage $n$ in the construction of $C$ such that $x,y\in I$, then $f_n(x)\ne f_n(y)$,
      because $f'_n(z)>0$ for all $z\in I$ by the definition of $f_n$. 
      If not, then find is the least stage $m$ and an interval $I$ such that $I$ is removed at 
      the $m$:th stage and $I$ is between $x$ and $y$.
      Then clearly again $f_m(x)\ne f_m(y)$. 
    \end{proofVOf}
    
    \begin{subclaim}\label{sbclinvlip}
      $(f\rest C)^{-1}$ is Lipschitz.
    \end{subclaim}
    \begin{proofVOf}{Subclaim \ref{sbclinvlip}}
      If $\eta\in 2^\o$, denote by $g(\eta)$ the unique point in $C$ which is obtained by going ``left'' at stage $n$
      if $\eta(n)=0$ and ``right'' if $\eta(n)=1$. That is, $g$ is the canonical
      homeomorphism of $2^\o$ onto $C$. It is not hard to see that
      $$g(\eta)=\sum_{n=0}^{\infty}\eta(n)\frac{2^{n+1}+6}{4^{n+1}}.$$
      Now $f_n(g(\eta))$ is the image of $g(\eta)$ under $f_n$ and by the definition of $f_n$
      we have $f_n(g(\eta))=1/\sqrt{2^{n}}$ if $\eta(n)=0$ and $f_n(g(\eta))=2/\sqrt{2^{n}}$ if $\eta(n)=1$; that is
      $f_n(g(\eta))=(1+\eta(n))/\sqrt{2^{n}}$.
      Let $\eta$ and $\xi$ be two arbitrary elements of $2^\o$, thus corresponding to the 
      two (arbitrary) elements $g(\eta)$ and $g(\xi)$ of $C$.
      Denote $c_n=|\eta(n)-\xi(n)|$. Note that for all $n\in\N$, $c_n^2=c_n$. Then
      \begin{eqnarray*}
        d(g(\eta),g(\xi))&=&\Big|\sum_{n=0}^\infty \eta(n)\frac{2^{n+1}+6}{4^{n+1}} - \sum_{n=0}^\infty \xi(n)\frac{2^{n+1}+6}{4^{n+1}}\Big|\\
        &=&\Big|\sum_{n=0}^\infty (\eta(n)-\xi(n))\frac{2^{n+1}+6}{4^{n+1}}\Big|\\
        &\le&\Big|\sum_{n=0}^\infty |\eta(n)-\xi(n)|\frac{2^{n+1}+6}{4^{n+1}}\Big|\\
        &=&\sum_{n=0}^\infty c_n \frac{2^{n+1}+6}{4^{n+1}}\\
        &=&\sum_{n=0}^\infty \frac{c_n}{\sqrt{2^n}} \cdot\frac{2^{n+1}+6}{2^{n+1}\sqrt{2^{n+1}}}\\ 
        \text{(Hölder) }&\le&\sqrt{\sum_{n=0}^\infty \frac{c_n}{2^n}} \cdot\underbrace{\sqrt{\sum_{n=0}^\infty\frac{2^{n+1}+6}{2^{n+1}\sqrt{2^{n+1}}}}}_{=:L}\\ 
        &=&L\cdot\sqrt{\sum_{n=0}^\infty \frac{c_n}{2^n}} \\
        &=&L\cdot\sqrt{\sum_{n=0}^\infty \Big|\frac{\eta(n)}{\sqrt{2^{n}}}-\frac{\xi(n)}{\sqrt{2^{n}}}\Big|^2} \\
        &=&L\cdot\sqrt{\sum_{n=0}^\infty \Big|\frac{1+\eta(n)}{\sqrt{2^{n}}}-\frac{1+\xi(n)}{\sqrt{2^{n}}}\Big|^2} \\
        &=&L\cdot\sqrt{\sum_{n=0}^\infty \Big|f_n(g(\eta))-f_n(g(\xi))\Big|^2} \\
        &=&L\cdot \|f(g(\eta))-f(g(\xi))\|_2.
      \end{eqnarray*}
      This verifies that the function $(f\rest C)^{-1}$ is Lipschitz.      
    \end{proofVOf}

    Since $C$ has positive measure, this implies that 
    the one-dimensional Hausdorff measure of $f[C]=((f\rest C)^{-1})^{-1}C$ must also have positive 
    measure. So it remains to show that
    $f[C]\subset H(T)$ and the proof of Claim~\ref{cl1} is done.

    \begin{subclaim}\label{sbclfcht}
      $f[C]\subset H(T)$.
    \end{subclaim}
    \begin{proofVOf}{Subclaim \ref{sbclfcht}}
      Suppose $\eta\in 2^\o$ and let $g(\eta)$ be as in the previous proof, 
      the canonical image of $\eta$ in $C$. Then, as above,
      $f_n(g(\eta))=(1+\eta(n))/\sqrt{2^{n}}$,
      so
      $$f(g(\eta))=\sum_{n=0}^\infty \frac{1+\eta(n)}{\sqrt{2^{n}}}e_{n,b(n)}.$$
      Now, by looking at the definition of $v_s$, one can see that
      the approximations of $f(g(\eta))$ of the form
      $$\sum_{n=0}^{k-1} \frac{1+\eta(n)}{\sqrt{2^{n}}}e_{n,b(n)}$$
      appear in $v_{b\restl k}$, so
      $f(g(\eta))\in \overline{\Cup_{s\in T}v_s}=H(T)$.
    \end{proofVOf}
  \end{proofVOf}

  \begin{claim}\label{cl2}
    If $T$ does not have a branch, then $H(T)$ is countable.
  \end{claim}
  \begin{proofVOf}{Claim~\ref{cl2}}
    If $H(T)$ is uncountable, then, because $\Cup_{s\in T}v_s$ 
    is countable, there is a point $x$ in 
    $\overline{\Cup_{s\in T}v_s}\setminus \Cup_{s\in T}v_s$. Let
    $(p_i)_{i\in\N}$ be a Cauchy sequence of elements of $\Cup_{s\in T}v_s$
    converging to $x$. By going to a subsequence, we can assume that for all $i\in\N$,
    $d(p_{i+1},p_i)<2^{-i}$. The latter inequality implies 
    by the definition of the sets $v_s$ that
    if $\dom s\le i$, then 
    $$p_i\rest\dom s\in v_s\iff p_{i+1}\rest\dom s\in v_s.$$ 
    So, we can
    find $b\in \o^\o$ such that $p_i\in v_{b\restl i}$ for all $i$ and
    so $(b\rest n)_{n\in\N}$ must be a branch in $T$.    
  \end{proofVOf}

  By Claims \ref{cl1} and \ref{cl2}, $T$ has no branch if and only if $H(T)$ is
  purely unrectifiable which concludes the proof.
\end{proofVOf}

\bibliography{ref}{}
\bibliographystyle{alpha}

\end{document}